\documentclass[reqno,12pt]{amsart} 
\usepackage{vmargin}
\setpapersize{A4}

\usepackage{amsmath} 

\usepackage{amsfonts}   

\usepackage{amssymb}      

\usepackage{eufrak}      

\usepackage{tikz}
\usetikzlibrary{matrix,calc}
\usepackage{tikz-cd}

\usepackage{amscd}      
\usepackage{amsthm}      

\usepackage{tikz,amsmath}
\usetikzlibrary{arrows}

\usepackage{epsfig}      

\usepackage{amstext}      

\usepackage{graphicx}

\usepackage[all,line,dvips]{xy} 
\CompileMatrices 

 \newcommand{\dist}{{\operatorname{dist}}}

\newcommand{\re}{\operatorname{Re}}
\newcommand{\im}{\operatorname{Im}}

\newcommand{\diam}{\operatorname{diam}}

 \newcommand{\supp}{\operatorname{supp}}

 \newcommand{\var}{\operatorname{var}}

   \theoremstyle{plain}
   \newtheorem{thm}{Theorem}[section]
   \newtheorem{prop}[thm]{Proposition}
   \newtheorem{lemma}[thm]{Lemma}  
   \newtheorem{cor}[thm]{Corollary}
   \theoremstyle{definition}
   
   \newtheorem{defn}[thm]{Definition}
   \newtheorem{example}[thm]{Example}
   \theoremstyle{remark}
   \newtheorem{obs}[thm]{Observation}
   \newtheorem{remark}[thm]{Remark}

\usepackage{mdframed,xcolor}

\definecolor{mybgcolor}{gray}{0.8}
\definecolor{myframecolor}{rgb}{.647,.129,.149}

\mdfdefinestyle{mystyle}{
  usetwoside=false,
  skipabove=0.6em plus 0.8em minus 0.2em,
  skipbelow=0.6em plus 0.8em minus 0.2em,
  innerleftmargin=.25em,
  innerrightmargin=0.25em,
  innertopmargin=0.25em,
  innerbottommargin=0.25em,
  leftmargin=-.75em,
  rightmargin=-0em,
  topline=false,
  rightline=false,
  bottomline=false,
  leftline=false,
  backgroundcolor=mybgcolor,
  splittopskip=0.75em,
  splitbottomskip=0.25em,
  innerleftmargin=0.5em,
  leftline=true,
  linecolor=myframecolor,
  linewidth=0.25em,
}

\newmdenv[style=mystyle]{important}

\usepackage{lipsum}

   \numberwithin{equation}{section}

        \date{\today}

\title[Conformal measures]{Dissipative conformal measures on locally compact spaces}  
\author{Klaus Thomsen}

\date{\today}

\email{matkt@imf.au.dk}
\address{Institut for Matematik, Aarhus University, Ny Munkegade, 8000 Aarhus C, Denmark}

\begin{document}

\maketitle

\begin{abstract} The paper introduces a general method to construct
  conformal measures for a local homeomorphism on a locally compact
  non-compact Hausdorff space, subject to mild irreducibility-like
  conditions. Among others the method is used to give necessary and
  sufficient conditions for the existence of eigenmeasures
  for the dual Ruelle operator associated to a locally compact non-compact
  irreducible Markov shift equipped with a uniformly continuous
  potential function. As an application to operator algebras the
  results are used to determine for which $\beta$ there
  are gauge invariant $\beta$-KMS weights on a simple graph
  $C^*$-algebra when the one-parameter automorphism group is given by a uniformly
  continuous real-valued function on the path space of the graph.

\end{abstract}

\section{Introduction}

A conformal measure for a discrete dynamical system made it first
appearance in the work of D. Sullivan, \cite{S}, in connection with
rational maps on the Riemann sphere, before the notion was coined and introduced
in a more general setting by M. Denker and M. Urbanski in
\cite{DU}. Conformal measures corresponding to various potential
functions now play important roles in the study of dynamical systems,
e.g. in holomorphic dynamics where they are used
as a tool to study the structure of the Julia set, among others. In
the setting of topological Markov shifts conformal measures arise
naturally via the thermodynamic formalism introduced by D. Ruelle, and
they constitute a key ingredient in the study of topological
Markov shifts with a countable number of states, \cite{Sa1},\cite{Sa2}. During the last
decade it has been realised that they also make their appearance in
connection with quantum statistical models that are based on
$C^*$-algebras and one-parameter group actions arising from local
homeomorphisms. This relation, which until recently only involved probability
measures on the dynamical system side and KMS states on the operator
algebra side, was extended in \cite{Th4} to a bijective correspondence
between general (possibly infinite) conformal measures and KMS weights. It is
this new connection between dynamical systems and $C^*$-algebras which
motivates the present study. Conformal measures are often required to
be probability measures, but for dynamical systems on non-compact
spaces, such as countable state Markov shifts for example, it is crucial to allow the measures to be
infinite. Nonetheless our knowledge of general, possibly infinite conformal measures is
lacking when it comes to the problem of determining the KMS weights in
the $C^*$-algebraic setting mentioned above, maybe because of reluctance to consider measures that are not conservative.

The
present paper seeks to improve on this by introducing a method to
construct conformal measures for a local homeomorphism on a locally compact
non-compact Hausdorff space which exploits the possibility of taking
limits along sequences that go to infinity. The method works in a
quite general non-compact setup and produces a non-zero fixed
measure for the dual Ruelle operator when the pressure of the potential is non-positive, and not
only when it is zero. More precisely, given a locally compact second countable
Hausdorff space $X$ and a local homeomorphism $\sigma : X \to X$, we
shall say that $(X,\sigma)$ is \emph{cofinal} when the equality
\begin{equation}\label{a25}
X = \bigcup_{i,j \in \mathbb N} \sigma^{-i}\left(
  \sigma^j\left(U\right)\right) 
\end{equation}
holds for every open non-empty subset $U \subseteq X$, and that
$(X,\sigma)$ is \emph{non-compact} when
\begin{equation}\label{c25}
\bigcup_{k=0}^{\infty} \sigma^k(U) 
\end{equation}
is not pre-compact (i.e. does not have compact closure) for any open
non-empty subset $U \subseteq X$. Assuming that $(X,\sigma)$ has these
two properties and given any continuous real-valued
function $\phi : X \to \mathbb R$, referred to in the following as the
\emph{potential}, there is a natural notion of
pressure $\mathbb P(\phi)$, and the method to be described produces regular non-zero $\phi$-conformal measures when $\mathbb
P(-\phi) \leq 0$. The measures can be finite or infinite, but they are  always
dissipative when $\mathbb P(-\phi) < 0$.

This construction is an extension of the
method by which positive eigenvectors of an infinite non-negative
matrix can be produced, and has its roots in the way harmonic functions are
constructed for Markov chains. See \cite{Th4} and \cite{Th5} for more details on this
connection. The extension we present here is also inspired by the PhD thesis of
Van T. Cyr, \cite{Cy}, where a similar method was used to produce
conformal measures for transient potentials with zero pressure on mixing
countable state Markov shifts. In that setting our results are more
complete because we obtain also a necessary condition for the
existence of a $\phi$-conformal measure: For a locally compact non-compact irreducible
Markov shift and a uniformly continuous potential $\phi$, a $\phi$-conformal measure exists if and only if $\mathbb P(-\phi) \leq 0$.

Besides the countable state Markov shifts a large and interesting
class of examples come from transcendental entire maps. When there are no critical
points in the Julia set the map is a local homeomorphism on its Julia
set, giving rise to a dynamical system which is both cofinal and
non-compact, possibly after the
deletion of a single exceptional fixed point. A prominent
example of this is given by the exponential family $\lambda e^z$ and in Section
\ref{exp} we consider the hyperbolic members of this family obtained by
choosing $\lambda$ real and between $0$ and $e^{-1}$, in order to show that the method
we introduce combines nicely with those developed
by Mayer and Urbanski in \cite{MU1}, \cite{MU2}, and gives rise to conformal
measures for the geometric potential $\log |z|$, not only when the
pressure function is zero, which by a result in \cite{MU1} occurs
exactly when the inverse temperature $t$ is equal to the Hausdorff
dimension $HD$ of the radial Julia set, but for all $t \geq HD$. The
additional conformal measures arise immediately from the general
results once we have shown that the notion of pressure employed here
agrees with that used by Mayer and Urbanski.

As indicated above our interest in conformal measures stems
from the bijective correspondence between $\beta
\phi$-conformal measures on $X$ and gauge invariant $\beta$-KMS
weights for the
one-parameter group arising from the
triple $(X,\sigma, \phi)$ by a well known canonical construction. In
the final section we give a brief introduction to this connection and summarise
the consequences of our results in the operator algebra setting. In
particular, they allow us to extend the results from \cite{Th4}
concerning the $\beta$-KMS weights of the gauge action on a simple graph
algebra to actions coming from an arbitrary uniformly continuous
potential. 




\section{Conformal measures}\label{Sec1}

Throughout the paper $X$ is a locally compact second countable Hausdorff space, $\sigma : X \to
X$ is a local homeomorphism and $\phi : X \to \mathbb R$ is a
continuous function.  

\begin{defn}\label{a38} A regular and non-zero Borel measure $m$ on
  $X$ is a \emph{$\phi$-conformal measure} when
\begin{equation}\label{a39}
m\left(\sigma(A)\right) = \int_A e^{\phi(x)} \ d m(x)
\end{equation}
for every Borel subset $A \subseteq X$ with the property that $\sigma
: A \to X$ is injective.
\end{defn} 

Let $C_c(X)$ be the space of continuous compactly supported functions
on $X$. We study conformal measures via the \emph{Ruelle operator}
$L_{\phi} : C_c(X) \to C_c(X)$ defined
such that
\begin{equation}\label{a47}
L_{\phi}(f)(x) = \sum_{y \in \sigma^{-1}(x)} e^{\phi(y)} f(y) .
\end{equation}
Since $L_{\phi} : C_c(X) \to C_c(X)$ is linear and takes non-negative
functions to non-negative functions it follows from the
Riesz representation theorem that $L_{\phi}$ defines a map
$m \mapsto L_{\phi}^*(m)$ on the set of regular Borel measures $m$ on $X$ by the
requirement that
\begin{equation*}\label{a40} 
\int_X f  \ d L_{\phi}^*(m) \ = \ \int_X L_{\phi}(f) \ dm
\end{equation*}
for all $f \in C_c(X)$.

\begin{lemma}\label{a101} Let $m$ be a non-zero regular Borel
  measure on $X$. The following are equivalent:
\begin{enumerate}
\item[1)] $m$ is ${\phi}$-conformal. 
\item[2)] $ L_{-\phi}^*(m) = m$.
\item[3)] When $U \subseteq X$ is an open subset such that $\sigma: U
\to X$ is injective and $g \in C_c(U)$,
$$
\int_{\sigma(U)} g \circ \left(\sigma|_U\right)^{-1} \ dm =  \int_U ge^{\phi}  \ d m.
$$
\item[4)] There is a covering $X = \bigcup_i U_i$ of $X$ by open
  subsets $U_i$ such that, for every $i$, $\sigma: U_i
\to X$ is injective and for all $g \in C_c(U_i)$,
$$
\int_{\sigma(U_i)} g \circ \left(\sigma|_{U_i}\right)^{-1} \ dm =  \int_{U_i} ge^{\phi}  \ d m.
$$

\end{enumerate}

\end{lemma}
\begin{proof} The equivalence between 1) and 2) was observed by
  Renault in \cite{Re2}, and a version of it appeared, in a slightly
  different setting, already in \cite{DU} where conformal measures
  were first introduced. Since the lemma is a crucial tool in the
  following, we present the proof here. We denote the characteristic function of a set $S$ by
$1_S$. 

1) $\Rightarrow$ 3): Let $U$ be as in 3). By linearity and continuity it suffices to
  establish the desired identity when $g = 1_A$ for some Borel
  subset $A \subseteq
  U$. In this case the identity is the same as (\ref{a39}).

3) $\Rightarrow$ 4) is trivial.

4) $\Rightarrow$ 2): By linearity it suffices to
show that
$$
\int_X L_{-\phi}(f) \ dm = \int_X f \ dm
$$
when $f$ is a supported in $U_i$. In that case the definition of
$L_{-\phi}$ shows that
$$
\int_X L_{-\phi}(f) \ dm = \int_{\sigma(U_i)}
e^{-\phi\left(\left(\sigma|_{U_i}\right)^{-1}(x)\right)} f\left(
  \left(\sigma|_{U_i}\right)^{-1}(x)\right) \ dm,
$$
which, thanks to 4) is equal to $\int_{U_i} f \ dm = \int_X f \ dm$.

2) $\Rightarrow$ 1): To establish (\ref{a39}) we write $A = \sqcup_i
A_i$ as a disjoint union of Borel sets $A_i$ such that for each $i$ there is an open and relatively compact subset $U_i \subseteq X$
where $\sigma$ injective, and $A_i \subseteq U_i$. Then 
$m\left(\sigma(A)\right) = \sum_i m\left(\sigma(A_i)\right)$
and
$$
\int_A e^{\phi(x)} \ dm(x) = \sum_i \int_{A_i} e^{\phi(x)} \ dm(x).
$$  
It suffices therefore to establish (\ref{a39}) when $A \subseteq U$
for some open and relatively compact subset $U \subseteq X$ where $\sigma$ is injective. Let
$V \subseteq U$ be open. It
follows from 2) that
$$
\int_{\sigma(V)}
e^{-\phi\left(\left(\sigma|_U\right)^{-1}(x)\right)} f\left(
  \left(\sigma|_U\right)^{-1}(x)\right) \ dm  = \int_V f \ dm,
$$
when $f \in C_c(V)$. Inserting for $f$ an increasing sequence from $C_c(V)$
converging up to $e^{\phi}1_V$, we find that (\ref{a39}) holds when
$A=V$ is an open subset of $U$. Since both measures,
$$
A \mapsto m\left(\sigma (A)\right) \ \text{and} \ A \mapsto \int_{A} e^{\phi} \
dm
$$
are finite Borel measures on $U$, they are also regular. It follows therefore that (\ref{a39}) holds for every
Borel subset $A$ of $U$.

\end{proof}




\section{Pressures associated with a cofinal local homeomorphism}

In this section we add first the assumption that $\sigma$ is cofinal,
i.e. for every open non-empty subset $U$ of $X$, the identity
(\ref{a25}) holds. We will use the notation $\|f\|$ for the supremum norm of
a bounded function $f$ on $X$, and we introduce the notation $\phi_n$
for the function
$$
\phi_n(x) = \sum_{j=0}^{n-1} \phi\left(\sigma^j(x)\right) .
$$
 
\begin{lemma}\label{b26} Assume that $(X,\sigma)$ is cofinal. Let $K$ be a compact subset of $X$ and $g \in
  C_c(X)$ a non-zero and non-negative function. There is an $N \in \mathbb N$ and a constant $C > 0$ such that
\begin{equation}\label{b27}
\left|L^n_{\phi}(f)(x)\right| \leq C\left\|f\right\|
\sum_{1 \leq i,j \leq N} L^{n+ j-i}_{\phi}(g)(x)
\end{equation}
for all $n > N$, all $x \in X$, and every function $f \in C_c(X)$ with
support in $K$.
\end{lemma}
\begin{proof} There is a $\delta > 0$ and an open non-empty subset $U
  \subseteq  X$ such that $g(x) \geq \delta$ for all $x \in U$. Since
  $\sigma$ is cofinal and $K$ compact there is an $N \in \mathbb N$ such that 
\begin{equation}\label{c27}
K \subseteq  \bigcup_{1 \leq i,j \leq N}
\sigma^{-i}\left(\sigma^j(U)\right)  .
\end{equation}
Set
\begin{equation*}
\begin{split}
&C_1 =  \sup \left\{ e^{-\phi_i(z)}: \ 1 \leq i \leq N,  \ z \in \supp g\right\} , \\
& C_2 =  \sup \left\{ e^{\phi_i(z)}: \
  1 \leq i \leq N,
  \ z \in K\right\}.
\end{split}
\end{equation*} 
Consider an element $f \in C_c(X)$ with support in
$K$ and let  $x \in X, n \in \mathbb N$, $n > N$. Since
$\sigma^{-n}(x) \cap K$ is a finite set, it follows from (\ref{c27})
that we can write $ \sigma^{-n}(x) \cap K$ as a disjoint union 
$$ \sigma^{-n}(x) \cap K  = \sqcup_{1 \leq i,j \leq N} B_{i,j},
$$
where $B_{i,j} \subseteq \sigma^{-i}\left(\sigma^j(U)\right)$ are finite sets. 
For $y \in B_{i,j}$ we choose $a_y \in U$ such that $\sigma^i(y) = \sigma^j(a_y)$. Then
\begin{equation}\label{c30}
L_{\phi}^n(f)(x) \ \ \ = \sum_{y \in \sigma^{-n}(x) \cap K} e^{\phi_n(y)}
f(y) = \sum_{1 \leq i,j \leq N} \sum_{y\in B_{i,j}}  e^{\phi_n(y)}
f(y) .
\end{equation}
For $y \in B_{i,j}$,
$$ 
e^{\phi_n(y)}|
f(y)| =  e^{\phi_i(y)}e^{\phi_{n-i}(\sigma^i(y))} |f(y)| =
e^{\phi_i(y)}e^{\phi_{n-i}(\sigma^j(a_y))}|f(y)|.
$$
Since $|f(y)| \leq \delta^{-1} \left\|f\right\|g(a_y)$, we find that
\begin{equation}\label{c33}
\begin{split}
&e^{\phi_n(y)}|
f(y)| \leq \delta^{-1}\left\|f\right\|
e^{\phi_i(y)}e^{\phi_{n-i}(\sigma^j(a_y))}g(a_y)\\
& =\delta^{-1} \|f\|
e^{\phi_i(y)}e^{-\phi_j(a_y)}e^{\phi_{n-i+j}(a_y)}g(a_y) \leq
\left\|f\right\| \delta^{-1}C_1C_2e^{\phi_{n-i+j}(a_y)}g(a_y) .
\end{split}
\end{equation}
To control the ambiguity of the association $y \mapsto
a_y$,
note that
$$
M_i = \sup_{y \in K} \# \left\{ z \in K : \ \sigma^i(z) = \sigma^i(y)
\right\}
$$
is finite for every $i \in \mathbb N$. Set $M = \max_{1 \leq i \leq N}
M_i$. Then
\begin{equation}\label{c31}
\begin{split}
\sum_{y \in B_{i,j}} e^{\phi_{n-i+j}(a_y)}g(a_y) \leq 
  M \sum_{ z \in
    \sigma^{-n-j+i}(x)}e^{\phi_{n-i+j}(z)}g(z)  = M L^{n-i+j}_{\phi}(g)(x).
\end{split}
\end{equation}
By combining (\ref{c30}), (\ref{c33}) and (\ref{c31}) we obtain
(\ref{b27}) if we set
$C = \delta^{-1}C_1C_2 M$.
\end{proof}

\begin{cor}\label{b28} Assume that $\sigma$ is cofinal. Let $f,g \in
  C_c(X)$ be non-zero and non-negative. It follows that
$$
\limsup_n \left(L^n_{\phi}(f)(x)\right)^{\frac{1}{n}} =  \limsup_n
  \left(L^n_{\phi}(g)(x)\right)^{ \frac{1}{n}}
$$
for all $x\in X$.
\end{cor}
\begin{proof} Let $K$ be a compact set containing the support of $f$ and
  let $C$ and $N$ be the numbers from Lemma \ref{b26}. If $\lambda > 0$ and $L^n_{\phi}(g)(x) \leq \lambda^n$
  for all $n \geq k$, it follows from (\ref{b27}) that for all $n \geq
  k + N $ there is a $j \in [n-N,n+N]$ such that
$$
\left( C \|f\| N^2\right)^{-1} L^n_{\phi}(f)(x) \leq \lambda^j .
$$
Thus 
$$
\left( C \|f\| N^2\right)^{-\frac{1}{n}}
  \left(L^n_{\phi}(f)(x)\right)^{\frac{1}{n}} \leq \alpha_n \lambda,
$$
for all large $n$,
where $\alpha_n = \max \left\{ \lambda^{\frac{j}{n} -1} : \ n-N \leq j
  \leq n+N \right\}$. Since $ \lim_{n \to \infty} \alpha_n =\lim_{n \to \infty} \left( C \|f\|
  N^2\right)^{-\frac{1}{n}} = 1$ we
conclude first that 
$\limsup_n
  \left(L^n_{\phi}(f)(x)\right)^{ \frac{1}{n}} \leq \lambda$ and then
  that 
$$
 \limsup_n
  \left(L^n_{\phi}(f)(x)\right)^{ \frac{1}{n}} \leq  \limsup_n
  \left(L^n_{\phi}(g)(x)\right)^{ \frac{1}{n}}.
$$
This argument shows that if one of the two limes superiors is finite
then so is the other, and by symmetry that they agree. Hence if one is
infinite, so is the other.
\end{proof}

In the following we denote by $C_c(X)^+$ the set of non-negative and
non-zero elements of $C_c(X)$. Using the convention that $\log 0 =
-\infty$ and $\log \infty = \infty$, it follows from Corollary
\ref{b28} that when $\sigma$ is cofinal,
we can define 
$$
\mathbb P_x(\phi)=  \log \left( \limsup_n
\left(L^n_{\phi}(f)(x)\right)^{\frac{1}{n}}\right) \  \in \ \left[-\infty,\infty\right],
$$
independently of which element $f \in
C_c(X)^+$ we use. We subsequently define the \emph{pressure} $\mathbb P(\phi)$ of
$\phi$ to be 
$$
\mathbb P(\phi) = \sup_{x \in X} \mathbb P_x(\phi) ,
$$
which is again an extended real number, i.e. $\mathbb P(\phi) \in [-\infty,\infty]$.




\subsection{The pressures associated to a cofinal Markov
  shift}\label{MS1}

In this section we relate the pressure defined above to the Gurevich
pressure known from topological Markov shifts. Because we can, and since
it is important for the applications to graph $C^*$-algebras, we work
in the same generality as in \cite{Th4} and \cite{Th5}, rather than restricting
the attention to irreducible or mixing Markov shifts.

Let $G$ be a directed graph with vertex set $V_G$ and edge
set $E_G$. We assume that both $V_G$ and $E_G$ are countable, and that
$G$ is 'row-finite', in the sense that the number
of edges emitted from any vertex is finite. Furthermore, we assume
that there are no sinks, i.e. every vertex emits an edge. 

An \emph{infinite path} in $G$ is an element $p = (p_i)_{i=1}^{\infty}  \in
\left(E_G\right)^{\mathbb N}$ such that $r(p_i) = s(p_{i+1})$ for all
$i$, where we have used the notation $r(e)$ and $s(e)$ for the range and source of an
edge $e \in E_G$, respectively.  A finite path $\mu = e_1e_2 \dots e_n$ is
defined similarly, and we extend the range and source maps to finite
paths such that $s(\mu) = s(e_1)$ and
$r(\mu) = r(e_n)$. The number of edges in $\mu$ is its \emph{length}
and we denote it by $|\mu|$. We let $\mathcal P(G)$ denote the set of
infinite paths in $G$ and extend the source map to $\mathcal P(G)$ such that
$s(p) = s(p_1)$ when $p = (p_i)_{i=1}^{\infty}$. To describe the
topology of $\mathcal P(G)$, let $\mu =
e_1e_2\cdots e_n$ be a finite path in $G$. We can then consider the \emph{cylinder} 
$$
Z(\mu) = \left\{p \in \mathcal P(G) : \ p_i = e_i, \ i =1,2, \dots,n
\right\}.
$$
$\mathcal P(G)$ is a totally disconnected second countable locally compact
Hausdorff space in the topology for which the collection of cylinders is
a base, \cite{KPRR}. The left shift $\sigma : \mathcal P(G) \to
\mathcal P(G)$ is the
map defined such that
$$
\sigma\left(e_0e_1e_2 e_3 \cdots\right) = e_1e_2e_3 \cdots .
$$ 
Note that $\sigma$ is a local homeomorphism. It is not difficult to
see that $(\mathcal P(G),\sigma)$ is cofinal if and only if $G$ is cofinal in the
sense introduced in \cite{KPRR}: If $v\in V_G$ and $p \in \mathcal P(G)$, there
is a finite path $\mu$ in $G$ and an $i \in \mathbb N$ such that
$s(\mu) = v$ and $r(\mu) = s(p_i)$. In the following we assume that
$G$ is a countable graph such that
\begin{enumerate}
\item[$\bullet$] $G$ is row-finite,
\item[$\bullet$] $G$ has no sinks, and
\item[$\bullet$] $G$ is cofinal.
\end{enumerate} 
For simplicity we summarise these conditions by saying that $G$ is
\emph{cofinal} when they hold. We denote by $NW_G$ the set of vertexes
$v$ which are
contained in a loop, meaning that there is finite path $\mu$ such that
$s(\mu) = r(\mu) = v$. These vertexes are called \emph{non-wandering},
and together with the edges they emit they constitute an irreducible
(or strongly connected) subgraph of $G$ which
we also denote by $NW_G$.

 For a continuous real-valued function
$\phi : \mathcal P(G) \to \mathbb R$ the \emph{Gurevich pressure}
$P_{NW_G}(\phi)$ of the restriction of $\phi$ to $\mathcal P(NW_G)$ is defined to be
\begin{equation}\label{c24}
\limsup_n \frac{1}{n} \log \sum_{\sigma^n(y)  =y} e^{\phi_n(y)} 1_{[v]}(y),
\end{equation} 
where $v$ is a vertex in $NW_G$ and $[v] = \left\{ x \in \mathcal P(G)
  : \ s(x) = v\right\}$, cf. e.g. \cite{Sa2}. It has been shown in increasing
generality by O.Sarig, \cite{Sa1}, \cite{Sa2}, that the 'limsup' above
is a limit which is independent of the choice of $v$ when $G$ is mixing (or primitive), and $\phi$ satisfies some
condition on its variation. In the general irreducible case, the
sequence involved in (\ref{c24}) will not converge, but the
independence of the choice of vertex is a general fact, at least when
$\phi$ is uniformly continuous in an appropriate metric. To make this
precise, set
$$
\var_k(\phi) = \sup \left\{\left|\phi(x) -\phi(y)\right|: \ x_i =
  y_i, \ i =1,2,\dots, k \right\} .
$$
We shall work with the assumption that
$\lim_{k \to \infty} \var_k (\phi) = 0$. Note that this condition
is implied by the uniform continuity of $\phi$ with respect to any
metric $d$ on $\mathcal P(G)$ with the property that for any $\delta  > 0$
there is a $k \in \mathbb N$ such that $x_i
= y_i, i = 1,2,\cdots,k  \ \Rightarrow \ d(x,y) \leq \delta$.

The next lemma follows from Proposition 3.2 in \cite{Sa2} when $G$ is
primitive and $\phi$ has the Walters property.

\begin{lemma}\label{c19} Assume that $G$ is cofinal and that $\lim_{k \to \infty} \var_k(\phi) =
  0$. The value of (\ref{c24}) does not depend on
  the choice of $v \in NW_G$, and for every finite path $\mu$ in $NW_G$,
$$
P_{NW_G}(\phi) = \limsup_n \frac{1}{n} \log \sum_{\sigma^n(y) = y}
e^{\phi_n(y)} 1_{Z(\mu)}(y) .
$$
\end{lemma}
\begin{proof} Let $v = s(\mu)$. Then
$$
\sum_{\sigma^n(y) = y}
e^{\phi_n(y)} 1_{Z(\mu)}(y)  \leq \sum_{\sigma^n(y) = y}
e^{\phi_n(y)} 1_{[v]}(y)
$$ 
for all $n$, and hence
\begin{equation}\label{c20}
 \limsup_n \frac{1}{n} \log \sum_{\sigma^n(y) = y}
e^{\phi_n(y)} 1_{Z(\mu)}(y)  \ \leq  \ \limsup_n \frac{1}{n} \log \sum_{\sigma^n(y) = y}
e^{\phi_n(y)} 1_{[v]}(y) .
\end{equation}
Let $w $ be any vertex in $NW_G$ and let $\epsilon > 0$. Since $NW_G$ is irreducible, there is a finite path $p$ in $G$ such that $s(p) = r(\mu), \ r(p)
= w$. By
assumption there is a $ k \in \mathbb N$ such that $\var_k(\phi) \leq
\epsilon$. Let $w_1,w_2, \cdots, w_l$ be the vertexes that can be reached from
$w$ by a path of length $k$. For each $w_i$ we choose a finite path
$q_i$ such that $s(q_i) = w_i, \ r(q_i) = s(\mu)$. Let $n > k$ and set
$$
M_i = \left\{y \in [w]: \ \sigma^n(y) = y, \ r(y_k) = w_i \right\},
$$
and define $\chi_i : M_i \to \left\{y \in Z(\mu) : \ \sigma^{n+L_i}(y)
  = y \right\}$, where $L_i = |\mu p| + |q_i|+ k$, such that
$\chi_i(y)$ is the infinite path which repeats the loop starting with
$\mu$, followed by $p$, the first $n+k$ edges in $y$ and ends with
$q_i$. In symbols, 
$$
\chi_i(y) = \left(\mu py_{[1,n+k]}q_i\right)^{\infty}  .
$$
By using that $\var_k(\phi) \leq \epsilon$ we find that
$$
e^{\phi_n(y)} \leq e^{n \epsilon} e^{\phi_n\left(\sigma^{|\mu p|}\left(\chi_i(y)\right)\right)} .
$$
By comparing $\phi_n\left(\sigma^{|\mu p|}\left(\chi_i(y)\right)\right)$ to
  $\phi_{n+L_i}\left(\chi_i(y)\right)$ one sees that
$$
 e^{\phi_n\left(\sigma^{|\mu p|}\left(\chi_i(y)\right)\right)} \leq
 C_ie^{\phi_{n+L_i}\left(\chi_i(y)\right)},
$$
where 
$$
C_i = \sup \left\{ e^{-\phi_{|q_i|+k}(z)} : \ z \in
  [w] \right\} \cdot \sup \left\{
  e^{-\phi_{|\mu p|}(z)} : \ z \in Z(\mu) \right\} .
$$
Since $\left\{y \in [w] : \ \sigma^n(y) = y\right\} = \sqcup_{i=1}^l
M_i$ and each $\chi_i$ is injective, this leads to the estimate
$$
\sum_{ \sigma^n(y) = y} e^{\phi_n(y)}1_{[w]}(y) \ \leq \ Ce^{n\epsilon} \sum_{i=1}^l
\ \ \sum_{\sigma^{n+L_i}(y) = y } e^{\phi_{n+L_i}(y)}1_{Z(\mu)}(y) ,
$$
where $C = \max_{1 \leq i \leq l} C_i$, and we conclude therefore that
$$
\limsup_n \frac{1}{n} \log \sum_{\sigma^n(y) = y}
e^{\phi_n(y)}1_{[w]}(y) \leq  \epsilon + \limsup_n \frac{1}{n} \log \sum_{\sigma^n(y) = y}
e^{\phi_n(y)}1_{Z(\mu)}(y) .
$$
Since $\epsilon > 0$ was arbitrary we can combine with (\ref{c20}) to
get
\begin{equation*}\label{c23}
\begin{split}
&\limsup_n \frac{1}{n} \log \sum_{\sigma^n(y) = y}
e^{\phi_n(y)}1_{[w]}(y) \leq  \limsup_n \frac{1}{n} \log \sum_{\sigma^n(y) = y}
e^{\phi_n(y)} 1_{Z(\mu)}(y)  \\
& \ \ \ \ \ \ \ \ \ \ \  \ \leq  \limsup_n \frac{1}{n} \log \sum_{\sigma^n(y) = y}
e^{\phi_n(y)} 1_{[v]}(y).
\end{split}
\end{equation*}
The desired conclusion follows from this by using the freedom in the choice of $\mu,v$ and $w$.

\end{proof}

\begin{prop}\label{c5} Let $X = \mathcal P(G)$ be the space of infinite paths
  in a cofinal graph $G$,
  and let $\sigma$ be the left shift on $\mathcal P(G)$. Let $\phi :
  \mathcal P(G) \to
  \mathbb R$ be a continuous function such that $\lim_{k \to \infty}
  \var_k(\phi) = 0$. Then
\begin{enumerate}
\item[a)] $\mathbb P(\phi) = -\infty$ when $NW_G = \emptyset$, and
\item[b)] $\mathbb P(\phi)$ is the Gurevich pressure $P_{NW_G}(\phi)$ of
  $\phi|_{\mathcal P(NW_G)}$ when
  $NW_G \neq \emptyset$.
\end{enumerate}
\end{prop}
\begin{proof} Let $x \in \mathcal P(G)$ and let $\epsilon > 0$ be
  arbitrary. Choose $k \in \mathbb N$ such that $\var_k(\phi) \leq
  \epsilon$ and let $\mu$ be the path of length $k$ with $x \in
  Z(\mu)$. If $x \notin \mathcal P(NW_G)$, the set $\sigma^{-n}(x)
  \cap Z(\mu)$ is empty for all $n$ and hence
  $L^n_{\phi}\left(1_{Z(\mu)}\right)(x) = 0$ for all $n$, leading to
  the conclusion that $\mathbb P_x(\phi) = - \infty$. Assume then that
  $x \in \mathcal P(NW_G)$. Let $n > k$. There is an
obvious bijection $\chi : \sigma^{-n}(x) \cap Z(\mu) \to \left\{ y \in
  Z(\mu) : \sigma^n(y) = y\right\}$ such that $\chi(z)_i = z_i, \ i =
1,2, \cdots, n+k$. Since $\var_k(\phi) \leq \epsilon$ this leads first
to the estimates 
\begin{equation}\label{h61}
\sum_{\sigma^n(y)=y} e^{\phi_n(y)- n\epsilon}1_{Z(\mu)}(y) \leq L^n_{\phi}\left(
  1_{Z(\mu)} \right)(x) \leq  \sum_{ \sigma^n(y)=y} e^{\phi_n(y)+
  n\epsilon}1_{Z(\mu)}(y),   
\end{equation}
and then by Lemma \ref{c19} to the conclusion that $P_{NW_G}(\phi) -
\epsilon \leq \mathbb P_x(\phi) \leq P_{NW_G}(\phi)+\epsilon$. Hence
$$
\mathbb P_x(\phi) = \begin{cases} - \infty, &  \ x \notin \mathcal
  P(NW_G) \\ P_{NW_G}(\phi) , & \ x \in \mathcal P(NW_G) . \end{cases}
$$
\end{proof}





\section{Constructing conformal measures}

In this section we return to the setting where $X$ is a second
countable locally compact Hausdorff space, $\sigma: X \to X$ is a
cofinal local homeomorphism and $\phi : X \to \mathbb R$ is
continuous. We add the assumption that $(X,\sigma)$
is \emph{non-compact}, in the sense that there is no open non-empty subset $U
\subseteq X$ such that $\bigcup_{n=0}^{\infty}\sigma^n(U)$ is
pre-compact. Since $(X,\sigma)$ is assumed cofinal, this additional
condition is satisfied when there is just a single point $x$ whose
orbit closure $\overline{\bigcup_{k=1}^{\infty} \phi^k(x)}$ is not
compact.

In the following we write $\lim_{k \to \infty} x_k = \infty$
when $\{x_k\}$ is a sequence in $X$ with the property that for any
compact subset $K \subseteq X$ there is an $N \in \mathbb N$ such that
$x_k \notin K \ \forall k \geq N$.

\begin{lemma}\label{a69} Assume that $(X,\sigma)$ is cofinal and 
  non-compact. Let $h \in C_c(X)^+$. There is a sequence $\{x_k\}$ in $X$ such that
$$
\sum_{n=0}^{\infty} L_{\phi}^n(h)(x_k)  > 0
$$
for all $k$ and $\lim_{x \to \infty} x_k = \infty$.
\end{lemma}
\begin{proof} Set $U = \left\{ x \in X :  \ h(x) > 0 \right\}$ and let
  $V_1 \subseteq V_2 \subseteq V_3 \subseteq \cdots$ be a sequence of
  open pre-compact subsets in $X$ such that $X = \bigcup_k V_k$. Since
  $(X,\sigma)$ is non-compact there is for every $k \in \mathbb N$ an element
$$
x_k \in \left( \bigcup_{n=0}^{\infty} \sigma^n\left(U\right)
  \backslash V_k \right).
$$
The sequence $\{x_k\}$ has the stated properties.
\end{proof}

A sequence $\{x_k\}$ with the properties stipulated in Lemma \ref{a69}
will be called \emph{$h$-diverging}.

\begin{lemma}\label{a16} Assume that $(X,\sigma)$ is cofinal and
  non-compact, and that 
$$
\sum_{n=0}^{\infty} \left|L_{\phi}^n(f)(x)\right| < \infty
$$
for all $x \in X$ and all $f \in C_c(X)$. Let $h \in C_c(X)^+$. For every $h$-diverging sequence $\{x_k\} \subseteq X$ there is a sub-sequence
$\left\{x_{k_i}\right\}$ such that the limit 
\begin{equation}\label{a71}
\lim_{i \to \infty} \ \frac{\sum_{n=0}^{\infty} L_{\phi}^n(f)(x_{k_i})}{\sum_{n=0}^{\infty}
L_{\phi}^n(h)(x_{k_i})}  
\end{equation}
exists for all $f \in C_c(X)$.
\end{lemma}
\begin{proof} Let $V_1 \subseteq V_2 \subseteq V_3 \subseteq \dots$ be
  a sequence of
  open relatively compact sets such that $X = \bigcup_i V_i$. Then
\begin{equation}\label{a62}
C_c(X) = \bigcup_i C_0(V_i),
\end{equation}
where $C_0(V_i)$ denotes the Banach space of continuous functions on
$V_i$ that vanish at infinity.
Since $C_0(V_i)$ is separable there is a sequence $\{g_n\}
\subseteq C_c(X)$ such that $ \left\{g_n\right\} \cap
  C_0(V_i) $ is dense in $C_0(V_i)$ for all $i$. It follows from Lemma
  \ref{b26} that for each $i$ there are $N_i \in \mathbb N$ and $C_i >0$
  such that 
\begin{equation*}\label{c39} 
\left|\sum_{n=0}^{\infty} L^n_{\phi}(f)(x)\right| \leq C_iN_i^2\left\|f\right\|
\sum_{n=0}^{\infty} L^n_{\phi} (h)(x) 
\end{equation*}
for all $f \in C_0(V_i)$ and all $x \notin \bigcup_{j=0}^{N_i}
\sigma^j\left(\overline{V_i}\right)$. Since $\lim_{k \to \infty} x_k =
\infty$ this implies that
\begin{equation}\label{a70}
 \left|\frac{\sum_{n=0}^{\infty} L_{\phi}^n(f)(x_k)}{\sum_{n=0}^{\infty}
L_{\phi}^n(h)(x_k)} \right| \leq C_iN_i^2 \left\|f\right\|
\end{equation}
for all $f \in C_0(V_i)$ and all sufficiently large $k \in \mathbb
N$. A diagonal sequence argument shows that there is a
sub-sequence $\{x_{k_i}\}$ such that 
$$
\lim_{i \to \infty} \frac{\sum_{n=0}^{\infty} L_{\phi}^n(g_j)(x_{k_i})}{\sum_{n=0}^{\infty}
L_{\phi}^n(h)(x_{k_i})}  
$$
exists for all $j$. Let $f \in C_0(V_i)$. It follows from (\ref{a70}) that
\begin{equation*}\label{c40}
\begin{split}
&\left|\frac{\sum_{n=0}^{\infty} L_{\phi}^n(g_j)(x_{k_l})}{\sum_{n=0}^{\infty}
L_{\phi}^n(h)(x_{k_l})}  \ - \ \frac{\sum_{n=0}^{\infty} L_{\phi}^n(f)(x_{k_l})}{\sum_{n=0}^{\infty}
L_{\phi}^n(h)(x_{k_l})}\right|  = \left|\frac{\sum_{n=0}^{\infty} L_{\phi}^n(g_j-f)(x_{k_l})}{\sum_{n=0}^{\infty}
L_{\phi}^n(h)(x_{k_l})}\right| \\
&\leq C_iN_i^2 \left\|g_j -f\right\|
\end{split}
\end{equation*}
for all large $l$ when $g_j \in C_0(V_i)$. Since $\left\{g_n\right\}
\cap C_0(V_i)$ is dense in $C_0(V_i)$ it  follows from this, in
combination with
(\ref{a62}), that the limit (\ref{a71})
exists for all $f \in C_c(X)$.
\end{proof}

\begin{lemma}\label{a17}   Assume that $(X,\sigma)$ is cofinal and
  non-compact, and that $\mathbb P(\phi) < 0$. Let $h \in  C_c(X)^+$. There
  is a regular Borel measure $m$ on $X$ such that $L_{\phi}^*(m) = m$
  and $\int_X h \ dm = 1$.
\end{lemma}
\begin{proof} It follows from the definition of $\mathbb P(\phi)$ that 
$\sum_{n=0}^{\infty} \left|L_{\phi}^n(f)(x)\right| < \infty$ for all
$x \in X$ and all $f \in C_c(X)$. By Lemma \ref{a16} and Lemma
\ref{a69} there is a sequence
  $\left\{x_k\right\}$ in $X$ such that $\lim_{k \to \infty} x_k =
  \infty$ and the limit
$$
\lim_{k \to \infty} \ \frac{\sum_{n=0}^{\infty} L_{\phi}^n(f)(x_{k})}{\sum_{n=0}^{\infty}
L_{\phi}^n(h)(x_{k})}  
$$
exists for all $f \in C_c(X)$. Riesz' representation theorem provides
us therefore with a regular Borel measure $m$ on $X$ such that
$$
\int_X f \ dm =\lim_{k \to \infty} \ \frac{\sum_{n=0}^{\infty} L_{\phi}^n(f)(x_{k})}{\sum_{n=0}^{\infty}
L_{\phi}^n(h)(x_{k})}  
$$
for all $f \in C_c(X)$. Note that
$$
\int_X f \ dm - \int_X L_{\phi}(f) \ dm = \lim_{k \to \infty} \ \frac{f(x_{k})}{\sum_{n=0}^{\infty}
L_{\phi}^n(h)(x_{k})}   \ = \ 0
$$
for all $f \in C_c(X)$ since $\lim_{k \to \infty} x_k =
  \infty$. 
\end{proof}

\begin{thm}\label{a63}   Assume that $(X,\sigma)$ is cofinal and non-compact,
  and that $\mathbb P(-\phi) \leq 0$. Let $h \in  C_c(X)$ be
non-negative and non-zero. There is a $\phi$-conformal measure $m$ such that $\int_X h \ dm
= 1$.
\end{thm}
\begin{proof}  Let $n \in \mathbb N$ and note that $\mathbb P\left(-\phi -
    \frac{1}{n}\right) = \mathbb P(-\phi) - \frac{1}{n} < 0$. Hence Lemma
  \ref{a39} and Lemma \ref{a17} give us a $\phi +
    \frac{1}{n}$-conformal measure $m_n$ for each $n \in \mathbb N$, with
  the additional property that $\int_X h \ dm_n =1$. Let $V_1
  \subseteq V_2 \subseteq \cdots$ be the sets from the proof of Lemma
  \ref{a16}. It follows from the way the
  $m_n$'s were constructed, in particular from (\ref{a70}), that there are numbers $M_k >
  0$, not depending
  of $n$, such that
\begin{equation}\label{d1}
\left|\int_X f \ d m_n \right| \leq M_k\|f\|
\end{equation}
for all $f \in C_0(V_k)$ and all $n$. (It is necessary here to check
that the constant $C$ in Lemma \ref{b26} can be chosen independently
of $n$.) Thus the sequence of linear functionals on $C_0(V_k)$
arising from integration with respect to the $m_n$'s are contained
in the ball of radius $M_k$ in the dual space of $C_0(V_i)$. By
compactness of this ball in the weak$^*$-topology we deduce that the
sequence has a convergent subsequence for each $k$. Combining
(\ref{a62}) with a diagonal sequence argument this leads to the
conclusion that there is a subsequence $\left\{m_{n_i}\right\}$ such
that the limit $\lim_{i \to \infty} \int_X f \ d m_{n_i}$ exists for
all $f \in C_c(X)$. By the Riesz representation theorem this gives us
a regular Borel measure $m$ on $X$ such that
$$
\lim_{i \to \infty} \int_X f \ dm_{n_i} = \int_X f \ dm 
$$   
for all $f\in C_c(X)$. In particular, $\int_X h \ dm =
1$. To check
that $m$ is ${\phi}$-conformal, let $U$ be an open subset of
$X$ such that $\sigma$ is injective on $U$. For each $i$ and each $g
\in C_c(U)$ we have that
$$
\int_{\sigma(U)} g\circ \left(\sigma|_U\right)^{-1}(x) \ dm_{n_i}(x) = \int_U g(x)
e^{\phi(x) + n_i^{-1}} \ dm_{n_i}(x) 
$$
by Lemma \ref{a101}. Since $g \in C_c(V_k)$ for some sufficiently large
$k$ and since $ge^{\phi + n_i^{-1}}$ converges uniformly to
$ge^{\phi}$, it follows from (\ref{d1}) that we can take the limit $i \to \infty$ to find that
$$
\int_{\sigma(U)} g\circ \left(\sigma|_U\right)^{-1}(x) \ dm(x) = \int_U g(x)
e^{\phi} \ dm(x) .
$$
Hence $m$ is ${\phi}$-conformal; again by Lemma \ref{a101}.
\end{proof}




\subsection{Conformal measures for cofinal Markov shifts}\label{MS2}

We return now to the setting of Section \ref{MS1}.

\begin{lemma}\label{c16}  Let $X = \mathcal P(G)$ be the space of infinite paths
  in a cofinal graph $G$,
  and let $\sigma$ be the left shift on $\mathcal P(G)$. Let $\phi :
  \mathcal P(G) \to
  \mathbb R$ be a function such that $\lim_{k \to \infty}
  \var_k(\phi) = 0$. Assume that there is a
  ${\phi}$-conformal
  measure $m$. Then $\mathbb P(-\phi) \leq 0$.
\end{lemma}
\begin{proof}  Thanks to Proposition \ref{c5} we can assume that $NW_G
  \neq \emptyset$. Let $v \in NW_G$ and $\epsilon > 0$ be arbitrary. There is a
$k \in \mathbb N$ such that $\var_k(\phi) \leq \epsilon$. Consider a finite
path $\mu$ in $NW_G$ of length $k$ with $s(\mu) = v$. Then
$$
\sum_{\sigma^n(y) = y} e^{-\phi_n(y)} 1_{Z(\mu)}(y) \leq 
\sum_{y \in \sigma^{-n}(x)} e^{-\phi_n(y) +n\epsilon} 1_{Z(\mu)}(y)
$$
for all $x \in Z(\mu)$ and all $n > k$, cf. (\ref{h61}). It follows that
\begin{equation}\label{c98}
\begin{split}
&m(Z(\mu)) \sum_{\sigma^n(y) = y} e^{-\phi_n(y)} 1_{Z(\mu)}(y)  \leq
e^{n \epsilon} \int_{Z(\mu)} L^n_{-\phi}\left(1_{Z(\mu)}\right) \ d m
\\
& \leq e^{n\epsilon} \int_{\mathcal P(G)} L^n_{-\phi}\left(1_{Z(\mu)} \right) \ d m = 
e^{n\epsilon} m(Z(\mu))
\end{split}
\end{equation}
for all large $n$.
Note that $m(Z(\mu)) > 0$ since $m$ is ${\phi}$-conformal and
$(X,\sigma)$ is cofinal. It follows therefore from (\ref{c98}) that
$$
\limsup_n \frac{1}{n} \log  \sum_{\sigma^n(y) = y} e^{-\phi_n(y)}
1_{Z(\mu)}(y) \leq \epsilon .
$$
Thanks to Lemma \ref{c19} and Proposition \ref{c5}, this completes the proof.
\end{proof}

To formulate the next theorem we need a stronger condition on $\phi$
when $NW_G$ is non-empty and finite. Following Walters, \cite{W}, we say
that $\phi$ \emph{satisfies Bowen's condition} on $NW_G$ when there is
a $C > 0$ such that
$$
\left| \sum_{i=0}^{n-1} \left[ \phi\left(\sigma^i(x)\right) -
 \phi\left(\sigma^i(y)\right)\right]\right| \leq C
$$
for all $(x_i)_{i=0}^{\infty} ,(y_i)_{i=0}^{\infty} \in \mathcal
P(NW_G)$ such that $x_i = y_i, i = 0,1,2, \cdots, n-1$, and all $n
\geq 1$.

\begin{thm}\label{c41} Assume $G$ is a cofinal graph and that
  $\phi : \mathcal P(G) \to \mathbb R$ is a function such that
  $\lim_{k \to \infty} \var_k(\phi) =0$.

\begin{enumerate}
\item[1)] Assume that $NW_G = \emptyset$. There is a
  ${\phi}$-conformal measure for the left shift on $\mathcal P(G)$.
\item[2)] Assume that $NW_G$ is non-empty and finite. Assume that
  $\phi$ satisfies Bowen's condition on $NW_G$. There is a
  ${\phi}$-conformal measure for the left shift on $\mathcal P(G)$ if and only if $\mathbb P(-\phi) = 0$,
  and it is then unique up multiplication by a scalar.
\item[3)] Assume that $NW_G$ is infinite.  There is an
  ${\phi}$-conformal measure for the left shift if and only if $\mathbb P(-\phi) \leq
  0$.
\end{enumerate}
\end{thm}
\begin{proof} It is easy to see that in case 1) and 3) there is an  element $x \in \mathcal P(G)$ such that $\lim_{j \to \infty}
  \sigma^j(x) = \infty$. Hence $(X,\sigma)$ is non-compact in these cases
  and the stated conclusions follow from a) of Proposition \ref{c5},
  Lemma \ref{c16} and Theorem \ref{a63}.

Consider then case 2). Let $p$ be the global period of $NW_G$. Then
$\mathcal P(NW_G)$ is the disjoint union of compact and open sets $X_i, i = 1,2,
\cdots ,p$, such that $\sigma(X_i) = X_{i+1}$, mod $p$. Furthermore,
the restriction of $\sigma^p$ to $X_p$ is a mixing subshift of finite
type. Since $\phi$ satisfies Bowen's condition on $NW_G$ by assumption,
it follows that $\phi_p$ satisfies Bowen's condition on $X_p$ with
respect to $\sigma^p$. It follows therefore from Theorem 1.3 and Theorem 2.16 in \cite{W} that there is a
$\phi_p$-conformal measure for $\sigma^p$ on $X_p$ if and only if the
pressure of the restriction of $-\phi_p$ to $X_p$ (with respect to
$\sigma^p$) is zero. It follows from b) of Proposition \ref{c5} that
this pressure is $p \mathbb P(-\phi)$. Since a
$\phi$-conformal measure for $\sigma$ on $\mathcal P(G)$ will
restrict to a $\phi_p$-conformal measure for $\sigma^p$ on $X_p$, we
deduce that there can only be a $\phi$-conformal measure for $\sigma$ on
$\mathcal P(G)$ if $\mathbb P(-\phi) = 0$. Assume then that $\mathbb
P(-\phi) = 0$. It follows from the theorems of Walters quoted
above that there
is a $\phi_p$-conformal measure for $\sigma^p$ on $X_p$, and that it is
unique up to multiplication by a scalar. 

It suffices now
to show that a $\phi_p$-conformal measure $\nu$ for $\sigma^p$ on $X_p$ extends
uniquely to a $\phi$-conformal measure $\mu$ for $\sigma$ on $\mathcal
P(G)$. For each $i < p$ there is a partition $X_i = \sqcup_j B_{i,j}$
of $X_i$ into compact and open sets
such that $\sigma^{p-i} : B_{i,j} \to X_p$ is injective for each
$j$. To extend $\nu$ to a $\phi$-conformal measure $\mu$ on $\mathcal
P(NW_G)$ we must therefore define $\mu$ on $X_i$ by the requirement that 
\begin{equation}\label{u1}
\int_{X_i} g \ d \mu = \sum_j \int_{\sigma^{p-i}(B_{i,j})} g \left(
  \left(\sigma^{p-i}|_{B_{i,j}}\right)^{-1}(x)\right)
e^{-\phi_{p-i}\left(\left(\sigma^{p-i}|_{B_{i,j}}\right)^{-1}(x)\right)}
  \   d \nu(x) 
\end{equation}
when $g \in C(X_i)$. In particular, we see that an extension of $\nu$ to a
$\phi$-conformal measure for $\sigma$ on $\mathcal P(NW_G)$ is unique. To see
that such an extension exists it is possible to show that the recipe
(\ref{u1}) provides the required extension by using that $\nu$ is
$\phi_p$-conformal on $X_p$. Alternatively, one can first extend $\nu$
to a measure $\nu$ on $\mathcal P(NW_G)$ such that $\nu(X_i) = 0, i
\neq p$, and then take
$$
\mu = \sum_{i=0}^{p-1} \left(\mathcal L_{-\phi}^*\right)^i(\nu),
$$
where $\mathcal L_{-\phi}$ denotes the compression of $L_{-\phi}$ to
$C(\mathcal P(NW_G))$, i.e.
$$
\mathcal L_{-\phi}(g)(x) \ = \sum_{y \in \sigma^{-1}(x) \cap \mathcal
  P(NW_G)} \ e^{-\phi(y)} g(y) 
$$ 
when $g \in C(\mathcal P(NW_G))$. To extend $\mu$ from $\mathcal
P(NW_G)$ to $\mathcal P(G)$ note that
\begin{equation}\label{c57}
\int_{\sigma([v])} g \circ \left(\sigma|_{[v]}\right)^{-1} \ d\mu \ =
\ \int_{[v]} ge^{\phi} \ d\mu 
\end{equation}
when $v \in NW_G$ and $g \in C([v])$. Set $H_0 = NW_G$ and 
$$
H_{n} = \left\{ w \in V_G : \ r\left(s^{-1}(w)\right) \subseteq H_{n-1}
\right\} 
$$ 
for $n\geq 1$. Then $H_n \subseteq H_{n+1}$ for all $n$ and $\bigcup_n H_n = V_G$ because $G$ is cofinal,
cf. the proof of Lemma 2.4 in \cite{Th5}. When $w \in  H_1$ we can define a regular
Borel measure $\mu_w$ on
$[w]$ by the requirement that
$$
\int_{[w]} f \ d\mu_w = \int_{\sigma([w])} e^{-\phi \left( \left(\sigma|_{[w]}\right)^{-1}(x)\right)} f \circ
\left(\sigma|_{[w]}\right)^{-1}(x) \ d\mu(x) 
$$
for all $f \in C([w])$.
We extend $\mu$ to a regular Borel measure on
$\left\{x \in \mathcal P(G) : \ s(x) \in H_1 \right\}$,
by setting
$$
\mu(B) = \mu\left(
  B \cap \mathcal P(NW_G)\right) + \sum_{w \in H_1 \backslash H_0} \mu_w(B\cap [w]).
$$
Then (\ref{c57}) holds for all $v \in H_1$. Continuing by induction
we get a regular Borel measure on $\mathcal P(G)$ such that
(\ref{c57}) holds for $v \in V_G$. It follows then from Lemma
\ref{a39} that $\mu$ is ${\phi}$-conformal, and it is clear from the
construction that it is the only
$\phi$-conformal measure extending $\nu$.

\end{proof}

\begin{remark} In the cases 1) and 3) of Theorem \ref{c41} a $\phi$-conformal measure is
  generally not unique. 

The reason for the introduction of Bowen's condition in case 2) is 
that we do not know (more precisely, the author does not know) if there can be a $\phi$-conformal measure on a
mixing one-sided subshift of finite type when $\mathbb P(-\phi) < 0$ and the potential
$\phi$ is continuous, but does not satisfy Bowen's condition. By
Theorem 2.16 in \cite{W} it is possible to use a condition slightly
weaker than Bowen's, but beyond that nothing seems to be known. When $\mathbb P(-\phi) = 0$
there \emph{is} a $\phi$-conformal measure, also when $\phi$ is only assumed to
be continuous. This follows from Theorem 6.9 in \cite{Th2} since the spectral
radius of the Ruelle operator is $1$ when $\mathbb P(\phi) = 0$ by
Theorem 1.3 in \cite{W}. It is, however, not clear if the measure is unique in
general. Therefore, without assuming Bowen's condition or the slightly
weaker
condition used by Walters in Theorem 2.16 of \cite{W}, the only thing
we can say in case 2) is that there is a $\phi$-conformal measure if
$\mathbb P(-\phi) = 0$, and none if $\mathbb P(-\phi) > 0$.

\end{remark}

For a mixing topological Markov shift O. Sarig has shown
the existence of an $e^{\mathbb P(\phi)}$-eigenmeasure for the dual
Ruelle operator when the potential has summable variation and is recurrent,
\cite{Sa1},\cite{Sa2}. Van Cyr extended this to transient potentials
in his thesis, \cite{Cy}. We can now supplement their results as
follows.

\begin{thm}\label{h20} Let $G$ be a countable connected directed graph
  with finite out-degree at each vertex. Assume that $G$ is not
  finite. Let $\phi : \mathcal P(G) \to
  \mathbb R$ be a function such that $\lim_{k \to \infty} \var_k (\phi)
  = 0$, and let $t \in \mathbb R$. There is a non-zero regular Borel measure $m$ on $\mathcal P(G)$ such that
$$
L_{\phi}^*(m) = e^t m 
$$ 
if and only if $t \geq \mathbb P(\phi)$.
\end{thm}
\begin{proof} $\mathcal P(G)$ is locally compact because $G$ has
  finite out-degree at each vertex, and second countable because $G$
  is countable. The dynamical system $(\mathcal P(G),\sigma)$ is
  cofinal because $G$ is connected, and non-compact because $G$ is
  also infinite. Therefore the
  theorem follows from Lemma \ref{a101} and Theorem \ref{c41} 3) after the
  observation that $\mathbb P(\phi - t) = \mathbb P(\phi) - t$.
\end{proof}

It follows from Sarigs results that there is a unique regular conservative
$e^t$-eigenmeasure for the dual Ruelle operator when $t = \mathbb P(\phi)$, provided $\phi$ is
recurrent, $G$ is aperiodic, and $\sum_{k=2}^{\infty} \var_k(\phi)
< \infty$, cf. \cite{Sa2}. His results do not require $G$ to have
finite out-degree at the vertexes. As will be shown in the next
section, at least in the locally
compact setting, $e^t$-eigenmeasures 
must be dissipative when $ t > \mathbb P(\phi)$, and when $t= \mathbb
P(\phi)$ and $\phi$ is transient.




\section{Dissipativity}\label{dis}

 In this section we assume only that $X$ is a second countable locally
compact Hausdorff space, $\sigma : X \to X$ is a local homeomorphism and
$\phi : X \to \mathbb R$ is continuous.

\begin{lemma}\label{a34}  Assume that $m$ is
  ${\phi}$-conformal. Then $m \circ \sigma^{-1}$ is absolutely
  continuous with respect to $m$.
\end{lemma}
\begin{proof} Write $X = \sqcup_{i \in \mathbb N} A_i$ as a disjoint
  union where the
  $A_i$'s are Borel subsets of $X$ such that
  $\sigma$ is injective on each $A_i$, and let $B \subseteq X$ be a
  Borel set. Since the
  conformality assumption implies that
\begin{equation*}\label{a35}
m(B \cap \sigma(A_i)) = m\left(\sigma\left(\sigma^{-1}(B) \cap
    A_i\right)\right) = \int_{\sigma^{-1}(B) \cap A_i} e^{\phi(x)}
\ dm(x),
\end{equation*}
it follows that $m\left(B\right) = 0 \ \Rightarrow \
  m\left(\sigma^{-1}(B) \cap A_i\right) = 0 \ \forall i \ \Rightarrow
  \ m(\sigma^{-1}(B)) = 0$. 
\end{proof}

In general $m \circ \sigma^{-1}$ is not equivalent to $m$; it is if
and only if $m(X \backslash \sigma(X)) = 0$. It follows from Lemma \ref{a34} that a ${\phi}$-conformal measure
$m$ gives rise to a \emph{Hopf decomposition} of $X$. That is, 
$$
X = C \sqcup D,
$$
where $C$ and $D$ are disjoint Borel sets with the following
properties, cf. \S 1.3 in \cite{Kr}:
\begin{enumerate}
\item[1)] $\sigma(C)  \subseteq C$,
\item[2)] For every Borel subset $A \subseteq C$,
$$
m \left( A \backslash \bigcap_{n=1}^{\infty} \bigcup_{k=n}^{\infty} \sigma^{-k}(A)\right) = 0.
$$
\item[3)] $D = \bigcup_{n=1}^{\infty} D_n$ where each $D_n$ is
  wandering in the sense that
$$
D_n \cap \sigma^{-k}(D_n) = \emptyset, \ k = 1,2,3, \cdots .
$$
\end{enumerate}
The set $C$ is the \emph{conservative part} of $\sigma$.

For a given ${\phi}$-conformal measure $m$, the Hopf decomposition
is unique modulo $m$-null sets, and we say that $m$
is \emph{dissipative} when the conservative part is an $m$-null set. Thus
$m$ is dissipative if and only if $X$ is the union of a countable collection
of wandering Borel sets, up to an $m$-null set.

\begin{lemma}\label{a45} Let $m$ be
  a ${\phi}$-conformal measure. Assume that $\sum_{n=1}^{\infty}
  L^n_{-\phi}(f)(x) < \infty$ for $m$-almost every $x$ when $f\in
  C_c(X)$ is non-negative. It follows that $m$ is dissipative.
\end{lemma}
\begin{proof} Let $C$ be the conservative part of $\sigma$. Assume for a contradiction that $m(C) >
  0$. Since $m$ is regular there is a non-negative $f \in C_c(X)$ such
  that 
$$
C_1 = \left\{ x \in C : \ f(x) \geq 1 \right\}
$$ 
has positive
  $m$-measure. Note that $m(C_1) < \infty$. Since $\sum_{k=1}^{\infty} L_{-\phi}^k(f)(x) < \infty$ for
  $m$-almost all $x$ by assumption, there is an $N \in \mathbb N$ such that
$$
C_2 = \left\{ x \in C_1 : \ \sum_{k=1}^{\infty} L^k_{-\phi}(f)(x) \leq N
\right\}
$$
has positive $m$-measure. Note that
\begin{equation}\label{d3}
Nm(C_2) \geq \int_{C_2} \sum_{k=1}^{\infty}
L^k_{-\phi}(f)(x) \ dm(x) =  \sum_{k=1}^{\infty} \int_X 1_{C_2}
L^k_{-\phi}(f) \ dm .
\end{equation}
Let $\{g_k\}$ be a uniformly bounded sequence from $C_c(X)$ which converges $m$-almost
everywhere to $1_{C_2}$. Then
\begin{equation*}\label{d4}
\begin{split}
& \int_X 1_{C_2}
L^k_{-\phi}(f) \ dm = \lim_{n \to \infty} \int_X g_n
L^k_{-\phi}(f) \ dm =  \lim_{n \to \infty} \int_X 
L^k_{-\phi}(f g_n \circ \sigma^k) \ dm \\
& = \lim_{n \to \infty} \int_X f g_n \circ \sigma^k \ dm =  \int_X f
1_{C_2}\circ \sigma^k \ dm \geq  \int_{C_2}  
1_{C_2}\circ \sigma^k \ dm .
\end{split}
\end{equation*}
Inserted into (\ref{d3}) we get that
\begin{equation}\label{d5}
 \int_{C_2}  
\sum_{k=1}^{\infty} 1_{C_2}\circ \sigma^k \ dm \leq Nm(C_2) < \infty.
\end{equation}
However, since $\sigma$ is infinitely recurrent on $C$ by 2) above, we know that 
$$
\sum_{k=1}^{\infty} 1_{C_2} \circ \sigma^k (x) = \infty
$$
for $m$-almost every $x \in C_2$. This contradicts (\ref{d5}) since
$m(C_2) > 0$.

\end{proof}

The following is an immediate consequence of Lemma \ref{a45}.

\begin{prop}\label{c43} Assume that $(X,\sigma)$ is cofinal and that
  $\mathbb P(-\phi) < 0$. It follows that every
  ${\phi}$-conformal measure for $\sigma$ is dissipative.
\end{prop}

We remark that Lemma \ref{a45} can also be used to show that for a
transient potential, as defined by O. Sarig in \cite{Sa1}, any
conformal measure is dissipative.

Although the conformal measures for a potential $\phi$ with $\mathbb
P(-\phi)$ negative must be dissipative, they may very well be finite. This
occurs already for constant potentials on certain locally compact
mixing Markov shifts.




\section{Conformal measures in the exponential family}\label{exp}

Let $h : \mathbb C \to \mathbb C$ be a holomorphic map and $J(h)$
the Julia set of $h$. We assume that $h$ is transcendental, i.e. is
not a polynomial, and then $J(h)$ is closed, unbounded and totally invariant under $h$, viz. $h^{-1}(J(h))
= J(h)$. If we assume that $h'(z) \neq 0$ for all $z \in J(h)$, it
follows  that $h$ is locally
injective on $J(h)$ and hence that $h : J(h) \to J(h)$ is a local
homeomorphism. By Montel's theorem, Theorem 3.7 in
\cite{Mi}, there is a set $\mathcal E(h) \subseteq \mathbb C$,
consisting of at most one point $x$, which must be totally
$h$-invariant in the sense that $h^{-1}(x) = \{x\}$, such that for any
open subset $U$ of $\mathbb C$ with $U \cap J(h) \neq \emptyset$,
$\bigcup_{i,j \in \mathbb N}^{\infty} h^{-j}\left(h^i(U)\right)   = \mathbb C \backslash \mathcal E(h)$.
It follows that $J(h)\backslash \mathcal E(h)$ is
totally $h$-invariant, locally compact in the relative topology and
that $h : J(h) \backslash \mathcal E(h) \to J(h)\backslash \mathcal
E(h)$ is cofinal. Another application of Montel's theorem shows that
$\left(J(h)\backslash \mathcal E(h),h\right)$ is also non-compact. Hence the results of the previous
sections apply to this dynamical system.

In this setting probability conformal measures have been constructed by Mayer and Urbanski in
\cite{MU1} and \cite{MU2} for a large class of entire functions when
the potential $\phi$ is chosen carefully. When $h$ comes from the
exponential family $h(z) = \lambda e^z$, the potential considered by
Mayer and Urbanski is  
\begin{equation}\label{g7}
\phi(z) = \log |z|,
\end{equation}
or some 'tame' perturbation of $\phi$. The inverse temperature $\beta$
for
which a finite $\beta \phi$-conformal measure exists is
invariably unique, but we can now show that at least for the hyperbolic members
of the exponential family $\lambda e^z$ the situation is very different when
infinite conformal measures are also considered. To substantiate this we assume that $h = E_{\lambda}$ where $E_{\lambda}(z) = \lambda
e^{z}$ for some $0 < \lambda <
e^{-1}$. In this case $\mathcal E(E_{\lambda}) = \emptyset$. 
It follows from Proposition 4.5 in \cite{MU1} that
$$
A_{\beta} : = \sup_{x \in J(E_{\lambda})}  \sum_{y \in
  E_{\lambda}^{-1}(x)} |y|^{-\beta} < \infty
$$
for all $\beta > 1$, a fact which is also easy to verify directly in
the present case. This implies, in particular, that we can define $L_{-\beta \phi}$ by
the same formula as above, viz.
$$
L_{-\beta \phi}(g)(x) = \sum_{y \in E_{\lambda}^{-1}(x)}  |y|^{-\beta}g(y),
$$
as a positive linear operator on the vector space of bounded functions on
$J(E_{\lambda})$ for all $\beta > 1$. In order to combine the methods
and results of this paper with those of \cite{MU1}, we only have to prove the
following lemma.

\begin{lemma}\label{g6} When $E_{\lambda} : J(E_{\lambda})
  \to J(E_{\lambda})$ for some $\lambda \in ]0,e^{-1}[$ and $\phi$ is
  the potential (\ref{g7}), 
$$
\mathbb P(- \beta\phi) = \limsup_n \frac{1}{n} \log L^n_{-\beta \phi}(1) (x)
$$
for all $x \in J(E_{\lambda})$ and all $\beta > 1$.
\end{lemma}
\begin{proof} It is shown in \cite{MU1} that $\limsup_n \frac{1}{n}
  \log L^n_{-\beta \phi}(1) (x)$ is independent of $x \in
  J(E_{\lambda})$. Since we clearly have the inequality
$$
\limsup_n \frac{1}{n} \log L^n_{-\beta \phi}(1) (x) \geq \limsup_n
\frac{1}{n} \log L^n_{-\beta \phi}(f) (x)
$$
for any non-negative $f \in C_c(J(E_{\lambda}))$ with $\|f\| \leq 1$, it suffices
therefore to find a single element $x_0 \in J(E_{\lambda})$ and a non-negative
function $f \in C_c(J(E_{\lambda}))$ such that
\begin{equation}\label{g10}
  \limsup_n
\frac{1}{n} \log L^n_{-\beta \phi}(f) (x_0) \geq \limsup_n \frac{1}{n}
\log L^n_{-\beta \phi}(1) (x_0) .
\end{equation}
For this purpose  we need the following observation which follows from
Lemma 5.3 in \cite{MU1}:

\begin{obs}\label{g3}  Set $B_R = \left\{z \in \mathbb C: \ |z| \leq
    R\right\}$. Then
$$
\lim_{R \to\infty} \sup_{x \in J(E_{\lambda})} \sum_{y \in E_{\lambda}^{-1}(x)
  \backslash B_R } |y|^{-\beta} \ = \ 0
$$
when $\beta > 1$.
\end{obs}

It is well-known that $J(E_{\lambda})$ contains a fixed point $x_0$
for $E_{\lambda}$. It follows
from Observation \ref{g3} that there is an $R > 0$ such that
 \begin{equation}\label{g8}
m  : = \sup_{x \in J(E_{\lambda})}  \sum_{y \in E_{\lambda}^{-1}(x) \backslash B_R
} |y|^{-\beta} < \left|x_0\right|^{-\beta} .
\end{equation}
Let $f \in C_c(J(E_{\lambda}))$ be a non-negative function such that $f(z) = 1$ when
$z \in B_R \cap J(E_{\lambda})$. Then $f + 1_{J(E_{\lambda})
  \backslash B_R} \geq 1$ on $J(E_{\lambda})$ and hence
\begin{equation}\label{f22}
L_{-\beta \phi}(f)(x) + L_{-\beta \phi}\left( 1_{J(E_{\lambda})
  \backslash B_R} \right)(x) \geq L_{-\beta \phi}(1)(x) 
\end{equation}
for all $x\in J(E_{\lambda})$.
It follows from (\ref{f22}) first that
$$
L_{-\beta \phi}(f)(x) \geq L_{-\beta \phi}(1)(x) - m, 
$$
and then that
$$
L^n_{-\beta \phi}(f)(x) \geq L^n_{-\beta \phi}(1)(x) -  mL^{n-1}_{-\beta \phi}(1)(x)  .
$$
for all $n$ and $x$. By using that
$$
L^n_{-\beta \phi}(1)(x_0) = \sum_{y \in
  E^{-1}_{\lambda}(x_0)}|y|^{-\beta} L^{n-1}_{-\beta \phi}(1)(y) \
\geq  \ 
\left|x_0\right|^{-\beta} L^{n-1}_{-\beta \phi}(1)(x_0),
$$
we obtain the estimate
$$
L^n_{-\beta \phi}(f)(x_0) \geq   \left(|x_0|^{-\beta} -m \right)  L^{n-1}_{-\beta \phi}(1)(x_0)  .
$$
for all $n$. Thanks to (\ref{g8}) this proves (\ref{g10}). 
\end{proof}

It follows from Lemma \ref{g6} that $\mathbb P(-\beta \phi) =
P(\beta)$, where $P$ is the pressure function considered
in Proposition 7.2 of \cite{MU1}. We can therefore deduce from
Proposition 7.2 in \cite{MU1} that there is a unique $\beta_0 > 1$ such that $\mathbb
P(-\beta_0 \phi) = 0$, and that $\mathbb P(-\beta \phi) < 0$ for all
$\beta > \beta_0$. As shown in \cite{MU1} the number $\beta_0$ is the
Hausdorff dimension $HD(J_r(E_{\lambda}))$ of the radial Julia set
$$
J_r(E_{\lambda}) = \left\{ z \in J(E_{\lambda}) : \ \liminf_n
  |E_{\lambda}^n(z)| < \infty \right\} .
$$  
In particular, $\beta_0 < 2$ by Corollary 1.4 in \cite{MU1}. A main
tool in \cite{MU1} for the study of $J_r(E_{\lambda})$ is a
${\beta_0 \phi}$-conformal measure. It follows now from Theorem
\ref{a63} that they exist for all $\beta \geq HD(J_r(E_{\lambda}))$,
and from Proposition \ref{c43} that they are dissipative when $\beta >
HD(J_r(E_{\lambda}))$, i.e. we have the following.

\begin{prop}\label{g11} For each $\lambda \in ]0,e^{-1}[$ and each
  $\beta \geq HD(J_r(E_{\lambda}))$ there is an ${\beta
    \phi}$-conformal measure for $E_{\lambda} : J(E_{\lambda}) \to
  J(E_{\lambda})$. For $\beta > HD(J_r(E_{\lambda}))$ the measures are dissipative.
\end{prop}

As shown in \cite{MU1} there is a Borel probability measure on
$J(E_{\lambda})$ which is a $\beta
  \phi$-conformal measure when $\beta =  HD(J_r(E_{\lambda}))$, and
in \cite{MU2} it is shown that it is unique. The paper \cite{MU2}
contains much more information on this Borel probability measure.

If we consider the potential
$$
\psi(z) = \log \left|E_{\lambda}(z)\right|,
$$
there is a bijective correspondence between the ${\beta
  \phi}$-conformal and the ${\beta \psi}$-conformal measures given
by sending the ${\beta \phi}$-conformal measure $m$ to the ${\beta
  \psi}$-conformal measure $|z|^{\beta} dm(z)$. It follows from the
argument in
Remark 5.7 in \cite{Th3} that there are no finite ${\beta
  \psi}$-conformal measure for any $\beta$, but Proposition \ref{g11}
shows that infinite conformal measures exist for all $\beta \geq HD(J_r(E_{\lambda}))$:

\begin{cor}\label{g12} For each $\lambda \in ]0,e^{-1}[$ and each
  $\beta \geq HD(J_r(E_{\lambda}))$ there is an ${\beta
    \psi}$-conformal measure for $E_{\lambda} : J(E_{\lambda}) \to
  J(E_{\lambda})$. For $\beta > HD(J_r(E_{\lambda}))$ the measures are dissipative.
\end{cor}




\section{KMS weights from conformal measures}\label{last}

A local homeomorphism $\sigma$ on a locally compact Hausdorff space
$X$ gives rise to an
\'etale groupoid $\Gamma_{\sigma}$ and hence a $C^*$-algebra
$C^*(\Gamma_{\sigma})$ by a construction introduced in
increasing generality by Renault, Deaconu and Anantharaman-Delaroche,
\cite{Re1}, \cite{De}, \cite{An}. To
describe the construction, set
$$
\Gamma_{\sigma} = \left\{ (x,k,y) \in X \times \mathbb Z  \times X :
  \ \exists n,m \in \mathbb N, \ k = n -m , \ \sigma^n(x) =
  \sigma^m(y)\right\} .
$$
This is a groupoid with the set of composable pairs being
$$
\Gamma_{\sigma}^{(2)} \ =  \ \left\{\left((x,k,y), (x',k',y')\right) \in \Gamma_{\sigma} \times
  \Gamma_{\sigma} : \ y = x'\right\}.
$$
The multiplication and inversion are given by 
$$
(x,k,y)(y,k',y') = (x,k+k',y') \ \text{and}  \ (x,k,y)^{-1} = (y,-k,x)
.
$$
$\Gamma_{\sigma}$ is a locally compact Hausdorff space in the topology where
open
sets $U,V$ in $X$ with the property that $\sigma^n$ is injective on $U$ and
$\sigma^m$ is injective on $V$ give rise to the open set
$$
\left\{ (x,k,y) \in U \times (n-m) \ \times V: \  \sigma^n(x) =
  \sigma^m(y) \right\}
$$
in $\Gamma_{\sigma}$, and where sets of this form constitute a base for the topology. The
set $C_c(\Gamma_{\sigma})$ of compactly supported functions on
$\Gamma_{\sigma}$ is a
$*$-algebra with product
$$
fg(x,k,y) = \sum_{z, \ n- m =k } f(x,n,z)g(z,m,y)
$$
and involution $f^*(x,k,y) = \overline{f(y,-k,x)}$. The $C^*$-algebra
$C^*(\Gamma_{\sigma})$ is the completion of $C_c(\Gamma_{\sigma})$
with respect to a natural norm, cf. e.g. \cite{An} or
\cite{Th1}. Since we have been working with the condition that $\phi$
is cofinal it is worth pointing out that cofinality of $\phi$ is
almost equivalent to the simplicity of $C^*(\Gamma_{\sigma})$. In
fact, by
Theorem 4.16 in \cite{Th1}, $C^*(\Gamma_{\sigma})$ is simple if and only
if $\sigma$ is cofinal (a condition which was called 'irreducibility' in
\cite{Th1}) and $\left\{ x \in X : \ \sigma^k(x) = x \right\}$ has
non-empty interior for all $k \in \mathbb N$. In particular, if
$(X,\sigma)$ is non-compact, $C^*(\Gamma_{\sigma})$ is simple if and
only if $\sigma$ is cofinal.

A continuous 
function $\phi : X \to \mathbb R$ gives rise to a continuous
one-parameter automorphism group $\alpha^{\phi}_t, \ t \in \mathbb R$,
on $C^*(\Gamma_{\sigma})$
defined such that
$$
\alpha^{\phi}_t(f)(x,k,y) = \lim_{n \to \infty} e^{i t \left(\phi_{k+n}(x) - \phi_n(y)\right)} f(x,k,y) 
$$
when $f \in C_c(\Gamma_{\phi})$. The case where $\phi$ is the constant function
$1$ yields the so-called \emph{gauge action}.

Any regular Borel measure $m$ on $X$ gives rise to a densely defined
lower semi-continuous weight $\varphi_m$ on $C^*(\Gamma_{\sigma})$ such
that
$$
\varphi_m(f) = \int_{X} f(x,0,x) \ dm(x)
$$
when $f \in C_c(\Gamma_{\phi})$. This construction is the link
between conformal measures and KMS-weights because it turns out that
for any $\beta \in \mathbb R$, a regular Borel measure $m$ is $\beta\phi$-conformal if and
only if the weight $\varphi_m$ is a $\beta$-KMS-weight for the one-parameter
group $\alpha^{\phi}$, cf. Proposition 2.1 and Lemma 3.2 in
\cite{Th4}. Furthermore, all gauge-invariant $\beta$-KMS weights arise
in this way by Proposition 3.1 in \cite{Th4}. Thanks to this relation
between KMS-weights and conformal measures the preceding methods and
results have consequences for KMS-weights, some of which we now
summarise. For example we get the following from Theorem \ref{a63}.

\begin{thm}\label{e2} Assume that $(X,\sigma)$ is non-compact and
  cofinal, and that $\phi : X \to \mathbb R$ is a continuous
  function. Let $\beta \in \mathbb R$ be a real number such that
  $\mathbb P(-\beta \phi) \leq 0$. Then there is a gauge invariant $\beta$-KMS weight for the
  one-parameter group $\alpha^{\phi}$ on $C^*(\Gamma_{\sigma})$.
\end{thm}

When $G$ is a cofinal graph, as those considered in Sections \ref{MS1}
and \ref{MS2}, the $C^*$-algebra $C^*(\Gamma_{\sigma})$ coming from
the shift $\sigma$ on $\mathcal P(G)$ is
known as the graph $C^*$-algebra associated with $G$, cf. \cite{KPRR},
and it is usually denoted by $C^*(G)$. 
From the results above regarding more general
potentials we easily 
get the following consequences.

\begin{thm}\label{e3} Let $G$ be a cofinal graph and $\phi :
  \mathcal P(G) \to \mathbb R$ a function such that $\lim_{k
    \to \infty} \var_k(\phi) =0$. Consider the one-parameter group
  $\alpha^{\phi}$ on $C^*(G)$.
\begin{enumerate}
\item[1)] Assume that $NW_G = \emptyset$. There is a gauge-invariant
  $\beta$-KMS weight for $\alpha^{\phi}$ for
  all $\beta \in \mathbb R$.
\item[2)] Assume that $NW_G$ is non-empty and finite. Assume that
  $\phi$ satisfies Bowen's condition on $NW_G$. There is a gauge-invariant
  $\beta$-KMS weight for $\alpha^{\phi}$ if and only if $\mathbb P(-\beta
  \phi) = 0$, and if it exists this $\beta$-KMS weight is 
  unique up to multiplication by a scalar.
\item[3)] Assume that $NW_G$ is infinite. There is a gauge-invariant
  $\beta$-KMS weight for $\alpha^{\phi}$ if and only if $\mathbb P(-\beta
  \phi) \leq 0$.
\end{enumerate}
\end{thm}
\begin{proof} Combine Theorem \ref{c41} above with Proposition 3.1 in \cite{Th4}. 
\end{proof}

It should be noted that in case 2) of Theorem \ref{e3} the function $\beta \mapsto
\mathbb P(-\beta \phi)$ may not have any zeroes, and hence KMS weights
(or states if $G = NW_G$) may not exist, cf. Example 3.7 in \cite{KR}.

\begin{cor}\label{e4}  Let $G$ be a cofinal graph and $\phi :
  \mathcal P(G) \to ]0,\infty[$ a function such that $\lim_{k
    \to \infty} \var_k(\phi) =0$. Consider the one-parameter group
  $\alpha^{\phi}$ on $C^*(G)$. 

\begin{enumerate}
\item[1)] Assume that $NW_G = \emptyset$. There is a gauge-invariant $\beta$-KMS weight for $\alpha^{\phi}$ for
  all $\beta \in \mathbb R$.
\item[2)] Assume that $NW_G$ is non-empty and finite.  Assume that
  $\phi$ satisfies Bowen's condition on $NW_G$. There is a
  $\beta_0 \in ]0,\infty[$ such that there is a gauge invariant $\beta$-KMS weight for
  $\alpha^{\phi}$ if and only if $\beta = \beta_0$, and it is then
  unique up to multiplication by a scalar.
\item[3)] Assume that $NW_G$ is infinite. Assume also that $\phi$ is bounded
  away from $0$ and $\infty$, i.e. there are $0 < a \leq b < \infty$
  such that $\phi(y) \in [a,b]$ for all $y\in \mathcal P(G)$. There is
  a $\beta_0 \in ]0,\infty]$ such that a gauge-invariant $\beta$-KMS weight
  for $\alpha^{\phi}$ exists if and only if $\beta \geq \beta_0$.
\end{enumerate}

\end{cor}
\begin{proof} Only 2) and 3) require proof. Consider first case 3) and
  assume that the Gurewich entropy $P_{NW_G}(0) = \mathbb P(0)$ of $NW_G$
  is infinite. For any $\epsilon > 0$, any $x \in \mathcal P(NW_G)$ and any $\beta \in \mathbb R$ there is a finite path $\mu$ in
  $NW_G$ such that
$$
\min \left\{e^{-\beta a}, e^{-\beta b}\right\}^ne^{-n\epsilon}
\sum_{\sigma^n(y) = y} 1_{Z(\mu)}(y) \ \leq
 \ L^n_{-\beta \phi}\left(1_{Z(\mu)}\right)(x)
$$  
for all large $n$, cf. (\ref{h61}). Since $e^{\mathbb P(0)} =
\limsup_n \left( \sum_{\sigma^n(y) = y}
  1_{Z(\mu)}(y)\right)^{\frac{1}{n}} = \infty$, it follows that
$\mathbb P_x(-\beta \phi) = \infty$. By Lemma \ref{c16} there are
therefore no $\beta \phi$ conformal measure and we have to set
$\beta_0 = \infty$ in this case. 

Case 2) and case 3) with $\mathbb P(0)$ finite can be handled
together. In both cases the proof depends
  on the finiteness and continuity of the function $\beta \mapsto
  \mathbb P(-\beta \phi)$. To establish these properties, note that there are
  constants $ 0 <  a
  \leq b < \infty$
  such that $a \leq \phi(x) \leq b$ for all $x\in \mathcal
  P(NW_G)$. In case 3) this is an assumption, and in case 2) it
  follows from compactness of $\mathcal P(NW_G)$. Let $\beta, \beta'
  \in \mathbb R, \ \beta \leq \beta'$, and consider a non-negative $f \in
  C_c(\mathcal P(G))$. Then 
$$
 \sum_{y \in
  \sigma^{-n}(x)} e^{(\beta - \beta')nb  } e^{-\beta
  \phi_n(y)}f(y) \leq \sum_{y \in \sigma^{-n}(x)} e^{- \beta' \phi_n(y)}f(y) \leq \sum_{y \in
  \sigma^{-n}(x)} e^{(\beta - \beta')n  a} e^{-\beta
  \phi_n(y)}f(y),
$$
leading to the estimates 
$$
(\beta - \beta') b + \mathbb P(- \beta \phi) \leq \mathbb P(-\beta' \phi) \leq  (\beta -
\beta') a + \mathbb
P(-\beta \phi) .
$$ 
It follows first that $\mathbb P(-\beta \phi) \in \mathbb
R$ for all $\beta \in \mathbb R$ since $\mathbb P(0)$ is finite; in case 2)
because $NW_G$ is and in case 3) by assumption. Once this is
established the estimates above show that $\beta \mapsto \mathbb P(-\beta \phi)$
is continuous, strictly decreasing and converges to $-\infty$ when
$\beta \to \infty$. In this way 2) and 3) follow from the
corresponding cases of Theorem \ref{e3}. 
\end{proof}

If we take $\phi$ to be constant $1$ in Corollary \ref{e4}, we recover
Theorem 4.3 in \cite{Th4}.

\end{document}